\documentclass[12pt]{amsart}

\usepackage{amscd}
\usepackage{amssymb,amsmath,amsthm}
\usepackage[all]{xy}
\usepackage[utf8]{inputenc}

\theoremstyle{plain}
\newtheorem{thm}{Theorem}[section]
\newtheorem{lem}[thm]{Lemma}

\newtheorem{cor}[thm]{Corollary}
\newtheorem{defn-lem}[thm]{Definition-Lemma}

\newtheorem{prop}[thm]{Proposition}

\theoremstyle{definition}
\newtheorem{defn}[thm]{Definition}

\newtheorem{rem}[thm]{Remark}

\def\rep{\operatorname{rep}}
\def\dim{\operatorname{dim}}

\def\Rok{\operatorname{Rok}}
\def\Ad{\operatorname{Ad}}
\def\id{\operatorname{id}}
\def\Aut{\operatorname{Aut}}

\def\Rok{\operatorname{Rok}}
\newcommand{\mc}{\mathcal}
\begin{document}
\title [A tracially sequentially-split $\sp*$-homomorphisms between $C\sp*$-algebras V]
       { On Dualities of actions}

\begin{abstract}
 We introduce the notion of the weak tracial approximate representability of a discrete group action on a unital  $C\sp*$-algebra which could have no projections like the Jiang-Su algebra $\mc{Z}$. Then we show a duality between the weak tracial Rokhlin property and the weak tracial approximate representability.   More precisely, when $G$ is a finite abelian group and $\alpha:G\to\Aut(A)$ is a group action on a unital simple infinite dimensional $C\sp*$-algebra, we prove that  
 \begin{itemize}
\item $\alpha$ has the weak tracial Rokhlin property if and only if $\widehat{\alpha}$ has the weak tracial approximate representability.
\item $\alpha$ has the weak tracial approximate representability if and only if $\widehat{\alpha}$ has the weak  tracial Rokhlin property.
\end{itemize}

\end{abstract}

\author { Hyun Ho \, Lee}

\address {Department of Mathematics\\
          University of Ulsan\\
         Ulsan, South Korea 44610 \\
  }
\email{hadamard@ulsan.ac.kr}


\date{}
\thanks{The first author's research was supported by Basic Science Research Program through the National Research Foundation of Korea(NRF) funded by the Ministry of Education(NRF-2018R1D1A1A01057489)\\
}
\maketitle


\section{Introduction}
The Rokhlin property for group actions on $C\sp*$-algebras, which was proposed as ``outerness condition'' by Kishimoto, can be regarded as a regularity condition, which is often used to show that various structural properties pass from a $C\sp*$-algebra to its crossed product. In spite of  some successes \cite{Izumi:finite, Kishi, Kishi1} it is hard to find the group action which possesses the Rokhlin property. In particular, in the case of finite group actions  the Rokhlin property is very rigid in the sense that it implies that the unit of the algebra can be written as a sum of projections indexed by the group. A less restrictive one named the tracial Rokhlin property was suggested by Phillips \cite{Phillips:tracial}, in which the Rokhlin projections are not required to add up to unity,  but just up to a small error in trace. It turns out that there are many examples of finite group actions with the tracial Rokhlin property \cite{Phillips:examples}. However, its defect is to demand the existence of abundant projections whenever we ask for and this restriction cannot be neglected because now we know that there are many projectionless unital $C\sp*$-algebras like the Jiang-Su algebra (see \cite{BJ}, \cite{JS} for instance). 

 In the case of finite group actions, an  earlier effort to overcome the flaw appeared in Hirshberg and Orovitz's  work in \cite{HO} and a different weakening of it was proposed  by Hirshberg and Phillips in \cite{HP}, in which they define the notion called Rokhlin dimension  replacing the single tower of projections with several towers consisting of positive contractions. The latter has been extended to the case of compact group actions by Gardella \cite{G:compact}. In this paper, we  look more closely Hirshberg and Orovitz's definition of  the generalized tracial Rokhlin property of a finite group action and based on it we suggest the weak tracial Rokhlin property even for a compact group action. With the dual group action in mind, we also define the weak tracial approximate representability for a discrete group action. For finite group actions, a similar notion was suggested recently in \cite{Ali:2021} based on a local approximation approach. Here we employ the ultrapower algebra of a given algebra and define it in term of a mapping property namely tracially sequentially-split by order zero \cite{LO:Orderzero}.   We refer the reader to  \cite{Ali:2021}, \cite{GHS} for examples of such actions. It turns out that these two notions complement each other; when a finite group $G$ is abelian, the main result allows one to  observe the weak tracial Rokhlin property (approximate representability) of an action $\alpha:G\curvearrowright A$ by looking at the weak tracial approximate representability (Rokhlin property) of its dual $\widehat{\alpha}: \widehat{G} \curvearrowright A\rtimes_{\alpha} G$ (see Theorem \ref{T:dualityofgroupaction}), thus we extend the previous result \cite [Theorem 4.14]{LO:Dualities} or \cite [Theorem 3.11]{Phillips:tracial} in a more general context. We can interpret our result in the sense of Gardella, Hirshberg, and Santiago; \emph{tracially} $\dim_{\Rok}(\alpha)=0$ if and only if \emph{tracially} $\dim_{\rep}(\widehat{\alpha})=0$ and \emph{tracially} $\dim_{\rep}(\alpha)=0$ if and only if \emph{tracially} $\dim_{\Rok}(\widehat{\alpha})=0$ (see \cite [Theorem 2.14] {GHS}).    
We think that its proof of the main result is also worthwhile to mention since a systematic approach is used rather than an \emph{ad hoc} method; for a compact group action $\alpha:G\to \Aut(A)$,  it is shown that $\alpha$ has the weak tracial Rokhlin property if and only if the map $1_{C(G)}\otimes \id_A: (A, \alpha) \to (C(G)\otimes A, \sigma\otimes \alpha)$ is $G$-tracially sequentially-split by order zero (see Corollary \ref{C:thefirstfactorembeddingsplit}) and for a discrete group action $\beta:H\to \Aut(B)$,  it is shown that $\beta$ has the weak tracial approximate representability if and only if $\iota_B: (B,\beta) \to (B\rtimes_{\beta}H, \Ad \lambda^{\beta})$ is $H$-tracially sequentially-split by order zero (see Corollary \ref{C:approxrepandmap}). This new concept is defined by the author and H. Osaka in \cite{LO:Orderzero} to provide the working framework in showing that several tracial regularity properties related the Elliott classification program pass from $B$ to $A$ when there is a $*$-homomorphism from $A$ to $B$ extending \cite{LO:Tracial} and \cite{BS}. In fact, the central idea is to show that the equivariant map $\phi:(A, \alpha) \to (B, \beta)$ is $G$-tracially sequentially-split by order zero if and only if  the equivariant map $\widehat{\phi}: (A\rtimes_{\alpha}G, \widehat{\alpha}) \to (B\rtimes_{\beta}G, \widehat{\beta})$ is $\widehat{G}$-tracially sequentially-split by order zero (see Theorem \ref{T:dualityoftraciallysequentiallysplitmap}).

\section{Tracially sequentially-split by order zero}
In this section we briefly review several notions and related results from \cite{LO:Orderzero}. 

Throughout the paper, we fix a free ultrafilter $\omega$ on $\mathbb{N}$ and recall that  
$l^{\infty}(\mathbb{N}, A)$ denotes  the $C\sp*$-algebra of  all bounded functions from $\mathbb{N}$  to  $A$. We define a closed ideal of $l^{\infty}(\mathbb{N}, A)$ as follows: 
\[c_{\omega}(\mathbb{N}, A)=\{(a_n)_n \in l^{\infty}(\mathbb{N}, A)  \mid \lim_{\omega}\|a_n \|=0 \}. \]
Then we denote by $A_{\omega}=l^{\infty}(\mathbb{N}, A)/c_{\omega}(\mathbb{N}, A)$ the ultrapower $C\sp*$-algebra of $A$  with respect to the filter $\omega$ that is  equipped with the norm  $\|a\|= \lim_{\omega} \|a_n\| $ for $a=[(a_n)_n] \in A_{\omega}$. In addition, we denote by $\pi_{\omega}$ the canonical quotient map from $l^{\infty}(\mathbb{N}, A)$  onto $A_{\omega}$. Note that we can embed $A$ into $A_{\omega}$ as constant sequences denoted by $\iota$, and we call $A_{\omega} \cap A'$ the central ultrapower algebra of $A$. For an automorphism of $\alpha$ on $A$, we also denote by $\alpha_{\omega}$ the induced automorphism on $A_{\omega}$ or $A_{\omega}\cap A'$ without confusion.  

We save the notation $\lesssim$ for the Cuntz subequivalence of two positive elements; for two positive elements $a, b$ in A $a$ is said to be Cuntz subequivalent to $b$, denoted by  $a \lesssim b$,  if there is a sequence $(x_n)$ in $A$ such that $\| x_nbx_n^*-a\| \to 0$ as $n\to \infty$. Often when $p$ is a projection, we see that $p \lesssim a$ if and only if  there is a projection in the hereditary $C\sp*$-subalgebra generated by $a$  which is Murray-von Neumann equivalent to $p$. For more details, we refer \cite{Cu, Ro:UHF1, Ro:UHF2} for example. 

Let $A$ and $B$ be $C\sp*$-algebras and let $\phi:A \to B$ be a completely positive map. It is said that $\phi$ has order zero if $\phi(a)\phi(b)=\phi(b)\phi(a)=0$ for $a,b \in A^{+}$ such that $ab=ba=0$.  We shall abbreviate  
a completely positive map of order zero as an order zero map. The following is a well-known characterization of an order zero map. 
\begin{thm}(Winter and Zachrias)
Let $A$ and $B$ be $C\sp*$-algebras and $\phi:A \to B$ be an order zero map. Let $C=C^*(\phi(A)) \subset B$. Then there are a positive element $\displaystyle h \in \mc{M}(C) \cap C^{\prime}$ with $\|h\|=\|\phi\|$ and a $*$-homomorphism 
 \[\pi_{\phi}:A \to \mc{M}(C)\cap \{h\}'\] such that for $a\in A$
\[ \phi(a)=h\pi_{\phi}(a).\]
If $A$ is unital, then $h=\phi(1_A) \in C$. 
\end{thm}

\begin{defn}\cite[Definition 3.1]{LO:Orderzero}\label{D:splitbyorderzero}
Let $A$ and $B$ be (unital) $C\sp*$-algebras. A $*$-homomorphism $\phi:A \to B$ is called tracially sequentially-split by order zero, if for every positive nonzero element $z \in A_{\omega}$  there exist an order zero map  $\psi: B \to A_{\omega}$ and a nonzero positive element $g\in A_{\omega}\cap A'$ such that 
\begin{enumerate}
\item $\psi(\phi(a))=ag$ for each $a\in A$, 
\item $1_{A_{\omega}} -g \lesssim z$ 
\end{enumerate}
\end{defn}

In reminiscence of  \cite{Lin:tracial, Lin:classification} we view $\psi\circ\phi$  equals to $\iota$ up to ``tracially small error''. The $\psi$ in the above definition is called a tracial approximate left-inverse of $\phi$. When both $\phi$ and $\psi$ are unital, then $g=\psi(1_B)$.  Although the diagram below is not commutative in a rigid sense, we still use it to symbolize that $\phi$ is tracially sequentially-split by order zero with its tracial approximate left inverse $\psi$;
\begin{equation*}\label{D:diagram}
\xymatrix{ A \ar[rd]_{\phi} \ar@{-->}[rr]^{\iota} && A_{\omega}. &\\
                          & B \ar[ur]_{\psi: \,\text{order zero.}} 
                          }
\end{equation*}
\begin{rem}
When the approximate tracial left inverse $\psi$ is  a $^*$-homomorphism, then we just say that $\phi$ is tracially sequentially-split \cite{LO:Dualities}. 
\end{rem}

The following proposition which is a tracial version of \cite[Lemma 2.4]{BS} will be used in Section \ref{S:tracial}. 
\begin{prop}\label{P:firstfactorembedding}
Let $C$ and $A$ be unital $C\sp*$-algebras.  Suppose that for any nonzero positive element $z\in A_{\omega}$ there exists an order zero  map $\phi: C \to A_{\omega}\cap A'$ such that  $1-\phi(1_C) \lesssim z$ in $A_{\omega}$.
Then the first factor embedding $\id_A\otimes 1_C:A \to A\otimes_{\max} C$ is tracially sequentially-split by order zero. Moreover, the converse is also true.
\end{prop}
\begin{proof}
Let us denote by $m$ the map from $A \otimes_{\max} (A_{\omega}\cap A')$ to $A_{\omega}$ which is defined by 
\[ a\otimes [(a_n)] \to [(aa_n)].\] 
For any nonzero positive element $z$ in $A_{\omega}$  we take a map $\phi:C \to A_{\omega}$ as above and define $\psi$ as the composition of two maps $m$ and $\id_A \otimes \phi$, i.e., $\psi=m\circ (\id_A \otimes \phi)$. It follows immediately that $\psi(1_A \otimes 1_C)=\phi(1_C)$.
Then it is easily checked that
\begin{enumerate}
\item $(\psi \circ (\id_A\otimes 1_C)(a)=\psi(a\otimes 1_C)=a\phi(1_C)=a\psi(1_A\otimes 1_C)$ for any $a\in A$,
\item $1-\psi(1_A\otimes 1_C)=1-\phi(1_C) \lesssim z$.
\end{enumerate}

Conversely,  for any positive nonzero element $x$ in $A_{\omega}$ we consider a tracial approximate inverse $\psi$ for $\id_A\otimes 1_C$. Let $\phi(c)=\psi(1_A\otimes c)$ for $ c \in C$. Obviously, $\phi$ is an order zero map from $C$ to $A_{\omega}$.  Write $\psi (\cdot)= h_{\psi}\pi_{\psi}(\cdot)$. 
\[\begin{split}
 \phi(c)a&=\psi(1_A\otimes c)a=h_{\psi}\pi_{\psi}(1_A\otimes c) a \\
 &=\pi_{\psi}(1_A\otimes c)h_{\psi} \pi_{\psi}(1_A\otimes 1_C)a \\
 &=\pi_{\psi}(1_A\otimes c) \psi(1_A\otimes1_C)a \\
 &=\pi_{\psi}(1_A\otimes c)\psi(a\otimes 1_C)=\psi(a\otimes 1_C)\pi_{\psi}(1_A \otimes c)\\
 &=a\psi(1_A\otimes 1_C)\pi_{\psi}(1_A \otimes c)=a \phi(c).
  \end{split}
 \]
Moreover, $1-\phi(1_C)=1-\psi(1_A\otimes 1_C) \lesssim x$. Thus we obtain the order zero map $\phi$ from $C$ to $A_{\omega}\cap A'$ such that $ 1-\phi(1_C) \lesssim x$.  
\end{proof}
\begin{rem}
In fact, the corresponding statement for the second factor embedding $1_C\otimes \id_A: A \to C\otimes_{\max}A$ will be used in this note.  
\end{rem}

When $A$ and $B$ are equipped with group actions,  instead of ordinary $*$-homomorphisms we consider equivariant ones to define the equivariant version of Definition \ref{D:splitbyorderzero}.

\begin{defn}\label{D:equivorderzero}
Let A and B be separable unital $C\sp*$-algebras and $G$ a locally compact group. Let $\alpha:G \to \Aut(A)$ and $\beta:G \to \Aut(B)$ be  continuous actions. An equivariant $*$-homomorphism $\phi : (A,\alpha) \to (B,\beta)$ is called  $G$-tracially sequentially-split by order zero if for every nonzero positive element $x$ in $A_{\omega}$ there exist an equivariant order zero map $\psi:(B,\beta) \to (A_{\omega}, \alpha_{\omega})$ and a positive element $g$ in $A_{\omega}\cap A'$ such that 
\begin{enumerate}
\item $\psi (\phi (a))=ga=ag$,  
\item $1_{A_{\omega}}-g \lesssim x$.
\end{enumerate} 
\end{defn}

 We shall use the following diagram to describe the equivariant map $\phi$ which is tracially sequentially-split by order zero and its tracial approximate left inverse $\psi$; 
\begin{equation*}\label{D:diagram}
\xymatrix{ (A, \alpha) \ar[rd]_{\phi} \ar@{-->}[rr]^{\iota} && (A_{\omega}, \alpha_{\omega}). &\\
                          & (B, \beta) \ar[ur]_{\psi: \, \text{order zero.}} 
                          } 
\end{equation*}
\begin{rem}
Again, if the tracial approximate left inverse  $\psi$ is an equivariant $^*$-homomorphism, then we say that $\phi$ is $G$-tracially sequentially-split. 
\end{rem}
The following is a straightforward generalization of Proposition \ref{P:firstfactorembedding} to the equivariant case. 
\begin{prop}\label{P:equivariantfirstfactorembedding}
Let $C$ and $A$ be unital $C\sp*$-algebras.  Let $\alpha:G \to \Aut(A)$ and $\beta:G \to \Aut(B)$ be two continuous actions of a compact group $G$. Suppose that for any nonzero positive element $x\in A_{\omega}$ there exists an equivariant orderr zero map $\phi: (C,\beta) \to (A_{\omega}\cap A', \alpha_{\omega})$ such that $1-\phi(1_C) \lesssim x$. Then the first factor embedding $\id_A\otimes 1_C:(A, \alpha) \to (A\otimes C, \alpha\otimes \beta)$ is tracially sequentially-split by order zero. Moreover, the converse is also true.
\end{prop}
\begin{proof}
The proof is almost same with Proposition \ref{P:firstfactorembedding}, the only thing to be careful is that the map $m: (A_{\omega}\cap A')\otimes A \to A_{\omega}$ in Proposition \ref{P:firstfactorembedding} is equivariant with respect to actions and this is easily checked. 
\end{proof} 
We can extend an order zero map between $C\sp*$-algebras to an order zero map between crossed products if it is equivariant with respect to group actions.   
\begin{lem}\cite[Lemma 2.5]{GHS}\cite[Proposition 4.6]{LO:Orderzero}\label{L:extension}
Let $G$ be a locally compact group, $A$ and $B$ be $C\sp*$-algebras, and $\alpha:G \to \Aut(A)$ and $\beta:G \to \Aut(B)$ continuous actions.  Given an equivariant  contractive order zero map $\phi: A \to B$, the expression 
\[(\phi\rtimes G) (\xi)(g)=\phi(\xi(g))\]
for $\xi \in L^1(G, A, \alpha)$ and $g \in G$, determines a contractive order zero map from $A\rtimes_{\alpha}G$ to $B\rtimes_{\beta}G$. 
\end{lem}

\section{Dualities of actions}\label{S:tracial}

Let us recall the following definition from \cite{HO} which is a generalization of  Phillips' tracial Rokhlin property of a finite group action in terms of positive contractions.    

\begin{defn}(Hirshberg and Orovitz)
Let $G$ be a finite group and $A$ be a simple infinite dimensional separable unital $C\sp*$-algebra. It is said  that $\alpha:G\to\Aut(A)$ has the generalized tracial Rokhlin property if  for every finite set $F \subset A$, every $\epsilon>0$, any nonzero positive element $x\in A$ there exist normalized positive contractions $\{e_g\}_{g\in G} \subset A$ such that 
\begin{enumerate}
\item $e_g \perp e_h$ when $g\neq h$,
\item $\| \alpha_g(e_h)-e_{gh} \| \le \epsilon$, \quad  for all $g, h \in g$,
\item $\|  e_ga -ae_g \| \le \epsilon$, \quad for all $g \in G, a \in F$,
\item  $1_A-\sum_{g} e_g \lesssim x$. 
\end{enumerate}
\end{defn}

We can reformulate the generalized tracial Rokhlin property of $$\alpha:G\to\Aut(A)$$ with exact relations using the ultrapower algebra. 

\begin{prop}\label{T:tracialRokhlinaction}
Let $G$ be a finite group and $A$ be a simple infinite dimensional separable unital $C\sp*$-algebra. Then $\alpha:G\to\Aut(A)$ has the generalized tracial Rokhlin property if  and only if for any nonzero positive element $x \in A_{\omega}$ there exist a mutually orthogonal  positive contractions $e_g$'s in $A_{\omega}\cap A'$ such that 
\begin{enumerate}
\item $\alpha_{\omega, g}(e_h)=e_{gh}$, \quad for all $g,h\in G$ where $\alpha_{\omega}:G\to\Aut(A_{\omega}) $ is the induced action,  
\item 
$1_{A_{\omega}}-\sum_{g\in G} e_g \lesssim x$.
\end{enumerate}
\end{prop}
\begin{proof}
The arguments are almost same with \cite[Theorem 3.3]{LO:Tracial} replacing projections with positive contractions. 
\end{proof}

We recall that the strict Rokhlin property of $\alpha: G \to \Aut(A)$ for a compact group $G$ can be rephrased in term of  the existence of a mapping  from $C(G)$ to $A_{\infty}$ the sequence algebra of $A$.   

\begin{defn}(cf. \cite{BS}, \cite{HW:Rokhlin})\label{D:Rokhlin}
Let $A$ be a separable unital $C\sp*$-algebra and $G$ a second countable, compact group. Let $\sigma:G\to\Aut(C(G))$ denote the canonical $G$-shift, that is, $\sigma_g(f) = f(g^{-1} \cdot)$ for all $f \in C(G)$ and $g\in G$. A continuous action $\alpha:G\to\Aut(A)$ is said to have the Rokhlin property if there exists a unital equivariant $*$-homomorphism
\[(C(G), \sigma)\to (A_{\infty}\cap A' , \alpha_{\infty}).\]
\end{defn}

The same perspective holds even for the case that $\alpha$ has the tracial Rokhlin property. But we emphasize that the target must be changed into the ultrapower $A_{\omega}$ since we don't want the tracial case to be reduced to the strict case (see \cite{LO:Orderzero} for more details). 

\begin{thm}\label{T:tracialRokhlinaction2}
Let $G$ be a finite group and $A$ a separable simple unital infinite dimensional $C\sp*$-algebra. Then $\alpha$ has the generalized tracial Rokhlin property if  and only if for every nonzero positive element $x$ in $A_{\omega}$ there exists a contractive equivariant order zero map $\phi$ from $(C(G), \sigma)$ to $(A_{\omega}\cap A', \alpha_{\omega})$ such that $1-\phi(1_{C(G)}) \lesssim x$ in $A_{\omega}$.
\end{thm}

\begin{proof}
``$\Longrightarrow$'': By Proposition \ref{T:tracialRokhlinaction}, for any nonzero positive $x\in A_{\omega}$ we can take mutually orthogonal positive contractions $e_g$'s in $A_{\omega}\cap A'$ such that $1-\sum_{g\in G} e_g \precsim x$. Then we define $\phi(f)=\sum_g f(g)e_g$ for $f\in C(G)$. Note that 
\[ \begin{split}
\phi(\bar{f} f) &= \sum_{g} \bar{f}(g) f(g) e_{g}\\
&=\left( \sum_{g\in G}f(g)e_{g}^{1/2} \right)^*\left(\sum_{h\in G} f(h)e_{h}^{1/2}\right)\geq 0
\end{split}\] since $e_he_g=0$ for $g\neq h$, and thus $\phi$ is positive. In addition, $\phi$ is completely positive since $C(G)$ is an abelian $C\sp*$-algebra. To show $\phi$ is order zero, suppose $f_1 \perp f_2$ in $C(G)$. Then again using the orthogonality of $e_g$'s
\[ 
\begin{split}
\phi(f_1) \phi(f_2) &=\left(\sum_{g} f_1(g)e_g\right)\left(\sum_h f_2(h) e_h\right)\\
&=\left( \sum_{g} f_1(g)f_2(g) e_g ^2 \right)=0
\end{split}
\]
Since $\{\delta_g\mid g\in G\}$ is a basis for $C(G)$, to show that  $\phi$ is equivarian  it is enough to check
\[ 
\phi(\sigma_g(\delta_h(\cdot)))=\phi(\delta_{gh}(\cdot))=e_{gh}=\alpha_{\omega, g}(e_h)=\alpha_{\omega, g}(\phi(\delta_h)).
\]
Finally, $1- \phi (1_{C(G)})= 1- \sum_{g\in G} e_g \lesssim x$. 

``$\Longleftarrow$'': Let $x$ be a nonzero positive element in $A_{\omega}$ and suppose that we have an equivariant order zero map $\phi:(C(G),\sigma) \to (A_{\omega}\cap A', \alpha_{\omega})$. Let $\delta_g$ be the characteristic function on a singleton $g$. Then for each $g\in G$ $e_g=\phi(\delta_{g})$ is a positive contraction in $A_{\omega}\cap A'$ such that $e_g \perp e_h$ for $g\neq h$ and $1-\sum_g e_g=1-\phi(1_{C(G)})\precsim x$. Moreover,  
\[\alpha_{\omega, g}(e_h)=\alpha_{\omega, g}(\phi(\delta_h))=\phi(\sigma_g (\delta_h))=\phi(\delta_{gh})=e_{gh},\] so we are done. 
\end{proof}

The following is a final form to view the generalized  tracial Rokhlin property of $\alpha:G\to\Aut(A)$. 
\begin{cor}\label{C:tracialRokhlinviasequentiallysplitmap}
Let $G$ be a finite group and $\alpha:G\to\Aut(A)$ an action of $G$ on a simple separable unital infinite dimensional $C\sp*$-algebra $A$. Then $\alpha$ has the generalized tracial Rokhlin property if and only if the map $1_{C(G)}\otimes \id_A : (A, \alpha) \to (C(G)\otimes A, \sigma \otimes \alpha)$ is $G$-tracially sequentially split by order zero. 
\end{cor}
\begin{proof}
It follows from Theorem \ref{T:tracialRokhlinaction2} and Proposition \ref{P:equivariantfirstfactorembedding}.
\end{proof}

In view of Definition \ref{D:Rokhlin} and Theorem \ref{T:tracialRokhlinaction2}, we propose the following definition of the weak tracial Rokhlin property for a compact group which can be interpreted as \emph{tracially} $\dim_{\Rok}(\alpha)=0$ (see \cite[Definition 2.3]{GHS} for the notation  $\dim_{\Rok}(\alpha)$).  
\begin{defn}\label{D:weaktracialRokhlin}
Let $G$ be a compact group, and $\alpha:G\to\Aut(A)$ an action of $G$ on a unital separable $C\sp*$-algebra $A$. Then we say that $\alpha$ has the weak tracial Rokhlin property if for every nonzero positive element $x$ in $A_{\omega}$ there exists a contractive equivariant order zero map $\phi$ from $(C(G), \sigma)$ to $(A_{\omega}\cap A', \alpha_{\omega})$ such that $1-\phi(1_{C(G)}) \lesssim x$ in $A_{\omega}$.
\end{defn}

\begin{cor} \label{C:thefirstfactorembeddingsplit}
Let $G$ be a compact group and $\alpha:G\to\Aut(A)$ an action of $G$ on a  unital separable $C\sp*$-algebra $A$. Then $\alpha$ has the weal tracial Rokhlin property if and only if the map $1_{C(G)}\otimes \id_A : (A, \alpha) \to (C(G)\otimes A, \sigma \otimes \alpha)$ is $G$-tracially sequentially split by order zero. 
\end{cor}
\begin{proof}
It follows from Definition \ref{D:weaktracialRokhlin} and Proposition \ref{P:equivariantfirstfactorembedding}.\end{proof}
To define and characterize the action dual to a compact group action with the weak tracial Rokhlin property, we  recall the tracial approximate representability of a finite abelian group action due to Phillips, and we present it based on \cite[Definition 4.8]{LO:Dualities} taking care of the modification from $A_{\infty}$  to $A_{\omega}$.

\begin{defn}\label{D:tracialapproximaterepresentability}
Let $G$ be a finite abelian group and $A$ be a simple unital separable infinite dimensional $C\sp*$-algebra. We say that $\alpha:G\to\Aut(A)$ is tracially approximately representable if for every positive nonzero element $z$ in $A_{\omega}$, there are a projection $e$ in $A_{\omega}\cap A'$ and a  unitary representation $u:G \to eA_{\omega}e$ such that 
\begin{enumerate}
\item $\alpha_{\omega, g}(eae)=u_g (eae)u_g^*$ in $A_{\omega}$\quad \text{for every $a$ in $A$},
\item $\alpha_{\omega, g}(u_h)=u_h$ for all $g,h\in G$,
\item $1-e$ is Murray-von Neumann equivalent to a projection $\overline{zA_{\omega}z}$.
\end{enumerate}
\end{defn}

Similarly we can characterize the tracial approximate representability of an action in terms of a tracailly sequentially-split map into the ultrapower. We denote the canonical unitary representation implementing $\alpha:G\to\Aut(A)$ by $\lambda^{\alpha}:G \to U(\mathcal{M}(A\rtimes_{\alpha}G))$.
  
\begin{thm}\cite[Proposition 4.8]{LO:Dualities}\label{T:approximaterepresentability}
Let $G$ be a finite abelian group and $A$ a simple unital separable infinite dimensional $C\sp*$-algebra. Then $\alpha:G\to\Aut(A)$ is tracially approximately representable if and only if for every nonzero positive element $z$ in $A_{\omega}$ there are a projection $e\in A_{\omega}\cap A'$ and an equivariant $*$-homomorphism $\psi:(A\rtimes_{\alpha}G, \Ad \lambda^{\alpha}) \to (A_{\omega}, \alpha_{\omega})$ such that 
\begin{enumerate}
\item $\psi(a)=ae$ for every $a\in A$, $\psi(\lambda^{\alpha}_g)=\omega_g$, 
\item $1-\psi(1_A)=1-e$ is Murray-von Neumann equivalent to a  projection in $\overline{zA_{\omega}z}$. 
\end{enumerate}
\end{thm}

\begin{cor}\cite[Corollary 4.9]{LO:Dualities}\label{C:traciallyapproxintermsofsequentiallysplit}
Let $G$ be a finite abelian group and $A$ be a unital separable simple infinite dimensional $C\sp*$-algebra.  Then $\alpha:G\to\Aut(A)$ is tracially approximately representable if and only if $\iota_A: (A, \alpha) \to (A\rtimes_{\alpha}G, \Ad(\lambda^{\alpha}))$ is $G$-tracially sequentially-split. 
\end{cor}

We believe that the right definition  which extends Definition \ref{D:tracialapproximaterepresentability} properly must induce the order zero map from $(A\rtimes_{\alpha}G, \Ad \lambda^{\alpha})$ to $(A_{\omega}, \alpha_{\omega})$ replacing $\psi$ in Theorem  \ref{T:approximaterepresentability}. 

\begin{defn}\cite[Definition 2.7]{GHS}
Let $G$ be a (locally) compact group, and $B$ be a $C\sp*$-algebra. It is said that a strongly continuous function $u:G \to B$ is an \emph{order-zero representation} of $G$ on $B$ if the following conditions are satisfied: 
\begin{enumerate}
\item $u_g$ is a normal contractions for all $g \in G$, and $u_{1_G}$ is positive,
\item $u_g u_h = u_1 u_{gh}$ for all $g, h \in G$, 
\item $u_g^*=u_{g^{-1}}$ for all $g\in G$. 
\end{enumerate}
\end{defn}

We first suggest the dual notion of the weak tracial Rokhlin property of a compact group action in term of strongly approximate innerness of a discrete group action. 
\begin{defn}\label{D:theweaktracialapproximaterepresentability}
Let $G$ be countable discrete  group and $A$ be a unital separable $C\sp*$-algebra. Then we say that  $\alpha:G\to\Aut(A)$ has the weak tracial approximate representability if for every nonzero positive element $z$ in $A_{\omega}$  there exist a positive nonzero contraction $e$ in $A_{\omega}\cap A'$ and an order zero representation $u: G \to A_{\omega}$ such that 
\begin{enumerate}
\item $u_{1_G}=e$, and $1-e \lesssim z$, 
\item $\alpha_{\omega, g}(ea)=u_g ea u_g^*$ for every $a \in A$ and $g \in G$, 
\item $\alpha_{\omega, g}(u_h)=u_{ghg^{-1}}$.
\end{enumerate} 
\end{defn}

As we characterize the weak tracial Rokhlin property of $\alpha$ in term of \emph{tracially} $\dim_{\Rok}(\alpha)=0$, we characterize the weak tracial approximate representability of $\alpha$ in term of \emph{tracially} $\dim_{\rep}(\alpha)=0$ (see \cite[Definition 2.10]{GHS} for the notation $\dim_{\rep}(\alpha)$). 

\begin{thm}\label{T:tracially dim_{rep}=0}
Let $G$ be countable discrete  group and $A$ be a unital $C\sp*$-algebra. Then $\alpha:G\to\Aut(A)$ has the weak tracial approximate representability if and only if for every nonzero positive element $z$ in $A_{\omega}$ there is an equivariant order zero map $\psi: (A\rtimes_{\alpha} G, \Ad \lambda^{\alpha}) \to (A_{\omega}, \alpha_{\omega})$ such that 
\begin{enumerate}
\item $\psi(a)=a\psi(1_A)$ for all $a\in A$ where $\psi(1_A) \in A_{\omega}\cap A'$, 
\item $1_{A_{\omega}}- \psi(1_A) \lesssim z$.  
\end{enumerate}
\end{thm}
\begin{proof}
For a nonzero positive element $z$ in $A_{\omega}$, consider an order zero representation $u: G \to A_{\omega}$ satisfying the conditions in Definition \ref{D:theweaktracialapproximaterepresentability}. In particular, $u_{1_G} \in A_{\omega}\cap A'$ and $1_{A_{\omega}} - u_{1_G} \lesssim z$.  Then by \cite[Proposition 2.8]{GHS} there is a unitary dilation $v: G \to A^{**}_{\omega}$ such that $v_g u_{1_G}=u_g=u_{1_G}v_g$ for all $g\in G$. Then we consider a covariant pair $(v , \iota_A)$. where $\iota_A: (A, \alpha) \to (A\rtimes_{\alpha}G, \Ad \lambda^{\alpha})$ is the natural embedding, which gives rise to a $^*$-homomorphism $v\rtimes \iota_A: A\rtimes_{\alpha} G \to A_{\omega}$ by \cite[Proposition 2.9]{GHS}. Note that this map sends $\sum_{g\in G} a_g \lambda_g^{\alpha}$ to $\sum_g a_g v_g$ and thus $v\rtimes \iota_A$ commutes with $u_{1_G}$. Now consider $\psi(\cdot)= v\rtimes \iota_A (\cdot)u_{1_G}$ which is an order zero map. First, note that  for all $a\in A$
\[
\psi(a)= au_{1_G}=a\psi(1_A)
\] 
where $\psi(1_A)=u_{1_G} \in A_{\omega}\cap A'$ and $1_{A_{\omega}}- \psi(1_A)= 1_{A_{\omega}} -u_{1_G} \lesssim z$.
  
We need to check whether $\psi$ is equivariant;  it is enough to consider an element $x$ of the form $a_h \lambda_h^{\alpha}$ (here $a_h \in A$) and check $\alpha_{\omega, g}(\psi(x))= \psi(\Ad \lambda_g^{\alpha}(x))$ for $g \in G$. Indeed, 
\[ \alpha_{\omega, g}(\psi(x))= \alpha_{\omega, g}(a_h u_h)=\alpha_g (a_h)u_h, \]
\[ \psi (\Ad \lambda_g^{\alpha}(x))= \psi(\alpha_g(a_h)\lambda_h^{\alpha})= \alpha_g(a_h)u_h.\]
Conversely, choose a nonzero positive element $z \in A_{\omega}$ and consider an equivariant order zero map $\psi: (A\rtimes_{\alpha} G, \Ad \lambda^{\alpha}) \to (A_{\omega}, \alpha_{\omega})$ such that 
\begin{enumerate}
\item $\psi(a)=a\psi(1_A)$ for all $a\in A$ where $\psi(1_A) \in A_{\omega}\cap A'$, 
\item $1_{A_{\omega}}- \psi(1_A) \lesssim z$.  
\end{enumerate}
Then we let $u_g=\psi(\lambda^{\alpha}_g)$, and it is not difficult to check $u$ is an order zero representation. 
Moreover, $\psi(1_A)= u_{1_G}$ and $1_{A_{\omega}}- u_{1_G} \lesssim z$.  The other conditions follow from plugging in  $x=a (\in A)$, $x=\lambda^{\alpha}_h$ separately into  $\alpha_{\omega, g}(\psi(x))= \psi(\Ad \lambda^{\alpha}_g(x))$. 

\end{proof}
Like the weak tracial Rokhlin property, the weak approximate representability is characterized in term of a tracially sequentially-split map by order zero. 
\begin{cor}\label{C:approxrepandmap}
 Let $G$ be countable discrete abelian group and $A$ a unital $C\sp*$-algebra. Then  $\alpha:G\to\Aut(A)$ has the weak tracial approximate representability if and only if the natural embedding $\iota_A: (A, \alpha) \to (A\rtimes_{\alpha}G, \Ad \lambda^{\alpha})$ is $G$-tracially sequentially-split by order zero. 
\end{cor}
\begin{proof}
It follows by Definition \ref{D:equivorderzero} and Theorem \ref{T:tracially dim_{rep}=0}.
\end{proof}

Now we turn to show  a duality result for actions of finite abelian group with the weak tracial Rokhlin property and  action of finite abelian group with the weak tracial approximate representability. This generalizes the duality result by Phillips \cite[Theorem 3.11]{Phillips:tracial}.  Essential ingredients are Corollay {C:thefirstfactorembeddingsplit} and Corollary {C:approxrepandmap}.  The duality result is actually obtained as an application of the general duality result Theorem \ref{T:dualityoftraciallysequentiallysplitmap}, which we show next. 

\begin{lem}\cite[Lemma 2.3]{HO}\label{L:HO}
Let $A$ be a simple, unital, non type $I$ $C\sp*$-algebra and let $n\in \mathbb{N}$. For every nonzero element $a$ in $M_n(A)$ there exists a nonzero positive elements $b\in A$ such that $1\otimes b \lesssim a$.
\end{lem}
\begin{lem}\label{L:stablesequentialsplit}
Let $G$ be a finite group and $\alpha:G\to\Aut(A)$ and $\beta:G\to\Aut(B)$  two actions on unital separable simple infinite dimensional $C\sp*$-algebras $A$ and $B$ respectively. Then $\phi:(A, \alpha) \to (B, \beta)$ is $G$-tracially sequentially-split by order zero if and only if $\phi\otimes \id_{M_n}: (A\otimes M_n, \alpha \otimes \rho) \to (B\otimes M_n, \beta\otimes \rho)$ is $G$-tracially sequentially-split by order zero for any $n \in \mathbb{N}$ where $\rho$ is an action $G$ on $M_n$.  

\end{lem}
\begin{proof}
 Suppose that $\phi:(A, \alpha) \to (B, \beta)$ is $G$-tracially sequentially-split by order zero.  Consider a nonzero positive element $z$ in $(A\otimes M_n)_{\omega} \cong A_{\omega}\otimes M_n$. 
Since $(A\otimes M_n)$ is also simple and non type $I$, there exists nonzero positive element $b \in A_{\omega}$ such that $b\otimes 1 \lesssim z$ by Lemma \ref{L:HO}. For this $b\neq 0$ we take a tracial approximate left inverse $\psi:(B,\beta) \to (A_{\omega}, \alpha_{\omega})$ and a positive contraction $g=\psi(1_B)\in A_{\omega}\cap A'$ such that
\begin{enumerate}
\item $\psi(\phi(a))=ag$, 
\item $1-g \lesssim b$.
\end{enumerate}  
To show $\phi\otimes \id_{M_n}$ is $G$-tracially sequentially-split by order zero, we consider $\psi\otimes \id_{M_n}: B\otimes M_n \to A_{\omega}\otimes M_n$ which is an order zero map by \cite[Corollary 3.3]{WZ}. 
Then $1_{A_{\omega}\otimes M_n} - (g\otimes 1)=(1-g)\otimes 1 \lesssim b\otimes 1\lesssim z$.  Note that $g \otimes 1$ is in $(A\otimes M_n)_{\omega}\cap (A\otimes M_n)'$ and $1_{A_{\omega}\otimes M_n}- \psi\otimes \id_{M_n}(1_{B\otimes M_n})= (1-g) \otimes 1$. Moreover, for $a\in A$
\[ \psi\otimes \id_{M_n}(\phi\otimes \id_{M_n}(a\otimes e_{ij}))=ag\otimes e_{ij}=(a\otimes e_{ij})(g\otimes 1),\]
where $\{ e_{ij}\}$ is a set of matrix units.  Therefore we have shown that $\phi\otimes \id_{M_n}$ is $G$-tracially sequentially-split. \\

Conversely, suppose that $\phi \otimes \id_{M_n}$ is $G$-tracially sequentially-split by order zero. Take any nonzero positive element $z\in A_{\omega}$, and consider a tracial approximate left inverse $\widetilde{\psi}$ corresponding to $z\otimes e_{11} \in A_{\omega}\otimes M_n$. Then we define $\psi:B \to A_{\omega}$ by the restriction of $\widetilde{\psi}$ to $B\otimes 1$. Note that  $(A_{\omega}\otimes M_n) \cap (A\otimes M_n)' \subset (A_{\omega}\otimes M_n) \cap (1\otimes M_n)'= A_{\omega} \otimes 1$ and $(A_{\omega}\otimes M_n) \cap (A\otimes M_n)' =(A_{\omega}\cap A')\otimes 1$. It follows that $\widetilde{\psi}(1)= g\otimes 1 $ where $g\in A_{\omega}\cap A'$. From  $1-\widetilde{\psi}(1)=(1-g)\otimes 1 \lesssim z\otimes e_{11}$,   we can deduce that $1-g \lesssim z$ in $A_{\omega}$. Also, we see that by viewing $A=A\otimes 1$
\[\psi(\phi(a))=\psi(\phi(a)\otimes1)=\widetilde{\psi}((\phi(a)\otimes1)=(a\otimes 1)(g\otimes1)=ag.\]
Thus, $\psi$ is a $G$-tracial approximate left inverse for $\phi$ corresponding to $z$.  
\end{proof}

We are ready to prove the following duality result between a $G$-tracially sequentially-split map by order zero and $\widehat{G}$-tracially sequentially-split map by order zero  for a finite abelian group $G$ as one of our main results which is an extended version of \cite[Theorem 4.13]{LO:Dualities}. 
\begin{lem}\cite[Lemma 5.1]{HO}\label{L:positive}
Let $\alpha:G\to\Aut(A)$ be an action of a discrete group $G$ on a simple, unital $C\sp*$-algebra $A$. Suppose that $\alpha_g$ is outer for all $g\in G\setminus \{1_G\}$. Then for every nonzero positive element $a$ in the reduced crossed product $A\rtimes_{\alpha, r} G$, there exists a nonzero element $b\in A$ such that $b\lesssim a$ in $A\rtimes_{\alpha, r}G $. 
\end{lem}
 
\begin{thm}\label{T:dualityoftraciallysequentiallysplitmap}
Let $G$ be a finite abelian group and $A$ and $B$  unital separable simple infinite dimensional $C\sp*$-algebras where $\alpha$ and $\beta$ act on respectively.  Further we assume that $\alpha:G\to\Aut(A)$ is an action  such that $A\rtimes_{\alpha}G$ is simple, in particular an outer action.  Then  the equivariant $*$-homomorphism $\phi:(A, \alpha) \to (B, \beta)$ is $G$-tracially sequentially-split by order zero if and only if $\widehat{\phi}=\phi\rtimes G:(A\rtimes_{\alpha}G, \widehat{\alpha}) \to (B\rtimes_{\beta}G, \widehat{\beta})$ is $\widehat{G}$-tracially sequentially-split by order zero. 
\end{thm}
\begin{proof}
Suppose that an equivariant $*$-homomorphism $\phi:(A, \alpha) \to (B, \beta)$ is $G$-tracially sequentially-split by order zero. Given a nonzero positive element $z$ in $(A\rtimes_{\alpha}G)_{\omega}$, by Lemma \ref{L:positive}  there is a nonzero positive element $a$ in $A_{\omega}$  such that $a \lesssim z$ in $(A\rtimes_{\alpha}G)_{\omega}$. Then we can take the equivarian tracial approximate left inverse $\psi:(B, \beta) \to (A_{\omega}, \alpha_{\omega})$ such that $\psi(\phi(x))=x\psi(1_B)$ for all $x$ in $A$ and $1-\psi(1_B) \lesssim a$ in $A_{\omega}$.  Let us denote the positive contraction $\psi(1_B)$ by $e$. Then 
\[1-e \lesssim a \lesssim  z.\] 
So consider the order zero map $\psi\rtimes G$ by Lemma \ref{L:extension}.
From the equivariantness of $\psi$, $e$ is invariant under the action of $\alpha_{\omega}$ so that  
\[ ((\psi \rtimes G)\circ(\phi \rtimes G))(a\lambda^{\alpha}_g) =(\psi \rtimes G)(\phi(a)\lambda^{\beta}_g)=\psi(\phi(a))\lambda_{g}^{\alpha_{\omega}}=(ae)\lambda_{g}^{\alpha_{\infty}}=
a\lambda^{\alpha_{\omega}}_g e. \]
Moreover,  via $A_{\omega} \hookrightarrow A_{\omega}\rtimes_{\alpha_{\omega}}G \hookrightarrow (A\rtimes_{\alpha}G)_{\omega}$
\[1-(\psi\rtimes G)(1)=1-\psi(1)=1-e.\]
\[ \widehat{\alpha}_{\omega, \gamma}((\psi\rtimes G) (a\lambda^{\beta}_g))=\gamma(g) \psi(a)\lambda^{\alpha_{\omega}}_g= (\psi\rtimes G)( \widehat{\beta}_{\gamma}(a\lambda_g^{\beta})).\] 
Thus we have shown that $\psi\rtimes G$ is the equivaraint tracial approximate inverse for $\phi \rtimes G$ and we are done.   

Conversely, suppose that the equivariant $*$-homomorphism $\widehat{\phi}=\phi \rtimes G:(A \rtimes_{\alpha}G, \widehat{\alpha}) \to (B \rtimes_{\beta}G, \widehat{\beta})$ is $\widehat{G}$-tracially sequentially-split by order zero. Note that by Takai duality \cite{T:duality} there are  equivariant isomorphisms, 
\[\kappa_A: (A\rtimes_{\alpha}G\rtimes_{\widehat{\alpha}}\widehat{G}, \widehat{\widehat{\alpha}}) \cong (A\otimes\mc{K}(l^2(G)), \alpha\otimes \rho)\] 
\[\kappa_B:(B\rtimes_{\beta}G\rtimes_{\widehat{\beta}}\widehat{G}, \widehat{\widehat{\beta}}) \cong (B\otimes \mc{K}(l^2(G)), \beta\otimes \rho)\] where $\rho$ is the $G$-action on the algebra of compact operators $\mc{K}(l^2(G))$ induced by the right-regular representation  for a compact abelian group $G$. 
Since $A$ is simple, we see that $\widehat{\alpha}: \widehat{G}\to \Aut(A\rtimes_{\alpha} G)$ is an action such that $A\rtimes_{\alpha} G \rtimes_{\widehat{\alpha}}\widehat{G}$ is simple.  By the claim that we have shown in the above we know that $\widehat{\widehat{\phi}}:( A\rtimes_{\alpha} G \rtimes_{\widehat{\alpha}}\widehat{G}, \widehat{\widehat{\alpha}}) \to ( B\rtimes_{\beta} G \rtimes_{\widehat{\beta}}\widehat{G}, \widehat{\widehat{\beta}})$ is $G$-tracially equivalently sequentially-split by order zero such that  the following  diagram is commutative:
\begin{equation}\label{D:Takai}
\xymatrix{ (A\rtimes_{\alpha}G\rtimes_{\widehat{\alpha}}\widehat{G}, \widehat{\widehat{\alpha}})  \ar[d]_{\kappa_A}\ar[r]^{\widehat{\widehat{\phi}}} & (B\rtimes_{\beta}G\rtimes_{\widehat{\beta}}\widehat{G}, \widehat{\widehat{\beta}}) \ar[d]^{\kappa_B}\\
                         (A\otimes M_n, \alpha\otimes \rho)\ar[r]^{\phi\otimes \id_{M_n}} & (B\otimes M_n, \beta\otimes \rho)}
\end{equation}
here $n$ is the cardinality of the group $G$ \cite[Remark 3.15]{BS}.  

Note that the isomorphism $\kappa_A$ induces the isomorphism denoted by $(\kappa_A)_{\omega}$ between $((A\rtimes_{\alpha}G\rtimes_{\widehat{\alpha}}\widehat{G})_{\omega}, (\widehat{\widehat{\alpha}})_{\omega}) $ and $((A\otimes M_n)_{\omega}, (\alpha\otimes \rho)_{\omega})$.
Now for  a nonzero positive element $z$ in $(A\otimes M_n)_{\omega}$ consider $(\kappa_A)_{\omega}^{-1}(z)=\tilde{z}$ in $(A\rtimes_{\alpha}G\rtimes_{\widehat{\alpha}}\widehat{G})_{\omega}$. Since $\widehat{\widehat{\phi}}$ is $G$-tracially sequentially-split by order zero, we have an order zero map $\widehat{\widehat{\psi}}$ from $(B\rtimes_{\beta}G\rtimes_{\widehat{\beta}}\widehat{G}, \widehat{\widehat{\beta}})$ to $((A\rtimes_{\alpha}G\rtimes_{\widehat{\alpha}}\widehat{G})_{\omega}, (\widehat{\widehat{\alpha}})_{\omega})$ such that 
\begin{enumerate}
\item $\widehat{\widehat{\psi}}(\widehat{\widehat{\phi}}(x))=x g$ for any $x\in A\rtimes_{\alpha}G\rtimes_{\widehat{\alpha}}\widehat{G}$ where $g=\widehat{\widehat{\psi}}(1) \in (A\rtimes_{\alpha}G\rtimes_{\widehat{\alpha}}\widehat{G})_{\omega} \cap (A\rtimes_{\alpha}G\rtimes_{\widehat{\alpha}}\widehat{G})'$,
\item $1-g \lesssim \tilde{z}$ in $(A\rtimes_{\alpha}G\rtimes_{\widehat{\alpha}}\widehat{G})_{\omega}$.
\end{enumerate}
Combining the map $\widehat{\widehat{\psi}}$ with (\ref{D:Takai}) we obtain the following diagram;
\begin{equation}\label{D:duality}
\xymatrix{ (A\rtimes_{\alpha}G\rtimes_{\widehat{\alpha}} \widehat{G}, \widehat{\widehat{\alpha}}) \ar[rd]_{\widehat{\widehat{\phi}}} \ar@{-->}[rr] \ar[dd]_{\kappa_A}&& ((A\rtimes_{\alpha}G\rtimes_{\widehat{\alpha}}\widehat{G})_{\omega}, (\widehat{\widehat{\alpha}})_{\omega})\ar[dd]^{(\kappa_A)_{\omega}}&\\
                          & (B\rtimes_{\beta}G\rtimes_{\widehat{\beta}}\widehat{G}, \widehat{\widehat{\beta}}) \ar[dd]^{\kappa_B } \ar[ur]_{\quad \widehat{\widehat{\psi}}: \text{order zero} } \\
                (A\otimes M_n, \alpha \otimes \rho) \ar[rd]_{\phi \otimes \id_{M_n}} \ar@{-->}[rr] && ((A\otimes M_n)_{\omega}, (\alpha \otimes \rho)_{\omega}) &\\
    &(B\otimes M_n, \beta\otimes \rho)  }          
                \end{equation}
Now consider the map $\tilde{\psi}=(\kappa_A)_{\omega} \circ \widehat{\widehat{\psi}} \circ (\kappa_B)^{-1}$ as suggested in (\ref{D:duality}). Then it is equivariant since all three maps are. For any $a\in A$, denoting by $\{e_{ij}\mid 1 \le i,j \le n\}$ the matrix units for $M_n$,   
\[
\begin{split}
(\tilde{\psi} \circ (\phi\otimes \id_{M_n}))(a\otimes e_{ij})&=((\kappa_A)_{\omega} \circ \widehat{\widehat{\psi}} \circ \kappa_B^{-1} \circ (\phi \otimes \id_{M_n}))(a \otimes e_{ij}) \\
&=((\kappa_A)_{\omega} \circ \widehat{\widehat{\psi}} \circ \widehat{\widehat{\phi}}\circ \kappa_A^{-1})(a\otimes e_{ij})\\
&=(\kappa_A)_{\omega}( \kappa_A^{-1}(a\otimes e_{ij})g)\\
&=(a\otimes e_{ij})((\kappa_A)_{\omega}(g)).
\end{split}
\] 
Moreover, 
$1-(\kappa_A)_{\omega}(g)=(\kappa_A)_{\omega}(1-g) \lesssim (\kappa_A)_{\omega} (\tilde{z})=z$. It follows  that $\tilde{\psi}$ is an order zero map which is the tracial approximate inverse to $\phi \otimes \id_{M_n}$  for a nonzero positive $z$ in $(A\otimes M_n)_{\omega}$. By Lemma \ref{L:stablesequentialsplit} we conclude that $\phi$ is $G$-tracially sequentially-split by order zero. 
\end{proof}

Then we prove the duality between the weak tracial Rokhlin property and the weak tracial approximate representability of actions.  We refer the reader to \cite[Theorem 3.11]{Ali:2021} for a different approach.   

\begin{lem}\cite[Proposition 4.25]{BS}\label{L:BS1}
Let $G$ a compact abelian group, $A$ a $C^*$-algebra and $\alpha:G\to\Aut(A)$ a continuous action.  Then there exists an equivariant isomorphism 
\[ \phi: (A\rtimes_{\alpha} G \rtimes_{\widehat{\alpha}} \widehat{G}, \Ad(\lambda^{\widehat{\alpha}})) \longrightarrow ((C(G)\otimes A)\rtimes_{\sigma\otimes \alpha} G, \widehat{\sigma\otimes \alpha}) \] making the following diagram commute 
\begin{equation*}
\xymatrix{ (A\rtimes_{\alpha}G, \widehat{\alpha}) \ar[d]_{(1_{C(G)}\otimes \id_A)\rtimes G} \ar[rr]^{\iota_{A\rtimes_{\alpha}G}} && ((A\rtimes_{\alpha}G) \rtimes_{\widehat{\alpha}}\widehat{G},\Ad(\lambda^{\widehat{\alpha}})) \\
                            ((C(G)\otimes A)\rtimes_{\sigma\otimes \alpha}G, \widehat{\sigma \otimes \alpha})\ar[urr]_{\phi^{-1}} }
\end{equation*} 
\end{lem}
\begin{lem}\cite[Proposition 4.26]{BS}\label{L:BS2}
Let $H$ be a discrete, abelian group. A a $C\sp*$-algebra and  $\beta:H\to\Aut(A)$ an action. There exists an equivariant $*$-isomorphism 
\[ \psi: (A\rtimes_{\beta} H \rtimes_{\Ad{(\lambda^{\beta}})}H, \widehat{\Ad(\lambda^{\beta}})) \longrightarrow (C(\widehat{H})\otimes (A\rtimes_{\beta} H), \sigma\otimes \widehat{\beta}) \] making the following diagram commute 
\begin{equation*}
\xymatrix{ (A\rtimes_{\beta}H, \widehat{\beta}) \ar[d]_{1\otimes \id_{A\rtimes_{\beta} H}} \ar[rr]^{\iota_A\rtimes H} && (A\rtimes_{\beta} H \rtimes_{\Ad{(\lambda^{\beta}})}H, \widehat{\Ad(\lambda^{\beta}})) \\
                           (C(\widehat{H})\otimes (A\rtimes_{\beta} H), \sigma\otimes \widehat{\beta}) \ar[urr]_{\psi^{-1}} }
\end{equation*} 
\end{lem}

\begin{thm}\label{T:dualityofgroupaction}
 Let $A$ be an infinite dimensional simple separable unital C*- algebra, and let $\alpha:G\to\Aut(A)$ be an action of a finite abelian group G on A such that $A\rtimes_{\alpha}G$ is also simple. Then
\begin{enumerate}
\item $\alpha$ has the weak tracial Rokhlin property if and only if  its dual action has the weak approximate representability.
\item $\alpha$ has the weak approximate representability if and only if its dual action has the weak tracial Rokhlin property.
\end{enumerate}
\end{thm}
\begin{proof}
(1): Suppose that $\alpha$ has the weak tracial Rokhlin property. Then by Corollary \ref{C:tracialRokhlinviasequentiallysplitmap}  the map $1_{C(G)}\otimes \id_A: (A, \alpha)\to(C(G)\otimes A, \sigma\otimes\alpha)$ is $G$-tracially sequentially-split by order zero. Theorem \ref{T:dualityoftraciallysequentiallysplitmap} implies that $(1_{C(G)}\otimes \id_A)\rtimes G: (A\rtimes_{\alpha}G, \widehat{\alpha}) \to ((C(G)\otimes A)\rtimes_{\sigma \otimes \alpha} G, \widehat{\sigma \otimes \alpha})$ is $\widehat{G}$-tracially sequentially-split by order zero. This means that for every nonzero positive element $z$ in $(A\rtimes_{\alpha}G)_{\omega}$ there are a positive contraction $g$ in the ultraproduct of $A\rtimes_{\alpha}G$ and a corresponding equivariant tracial approximate inverse $\psi$ from $((C(G)\otimes A)\rtimes_{\sigma\otimes\alpha} G, \widehat{\sigma \otimes \alpha})$ to $((A\rtimes_{\alpha}G)_{\omega}, (\widehat{\alpha})_{\omega})$ such that $\psi(1_{C(G)\otimes A})=g$ and $1-g \lesssim z$. So we have the diagram below; 

\begin{equation*}\label{D:diagram1}
\xymatrix{ (A\rtimes_{\alpha}G, \widehat{\alpha}) \ar@{-->}[rr]_{\iota} \ar[rd]_{(1_{C(G)}\otimes \,\id_A)\rtimes G \quad} && ((A\rtimes_{\alpha} G)_{\omega}, \widehat{\alpha}_{\omega} ).&\\
                          & ((C(G)\otimes A)\rtimes_{\sigma\otimes \alpha}G, \widehat{\sigma \otimes \alpha})  \ar[ur]_{\psi: \text{order zero} }}
\end{equation*} 
Then combining this diagram with the diagram in Lemma \ref{L:BS1} we obtain the following diagram
 \begin{equation*}\label{D:diagram2}
\xymatrix{ (A\rtimes_{\alpha}G, \widehat{\alpha}) \ar[dd]_{\iota_{A\rtimes_{\alpha}G}}\ar[rd]_{(1_{C(G)}\otimes \,\id_A)\rtimes G \quad} \ar@{-->}[rr]^{ \iota} &&((A\rtimes_{\alpha}G)_{\omega}, (\widehat{\alpha})_{\omega})  &\\
                          & ((C(G)\otimes A)\rtimes_{\sigma\otimes \alpha}G, \widehat{\sigma \otimes \alpha})  \ar[ur]_{\psi: \,\text{order zero}} \\
  ((A\rtimes_{\alpha}G) \rtimes_{\widehat{\alpha}}\widehat{G},\Ad(\lambda^{\widehat{\alpha}})) \ar[ru]_{\phi} }
\end{equation*} 
Naturally, we consider $\widetilde{\psi}= \psi \circ \phi$ as a tracial approximate inverse for $\iota_{A\rtimes_{\alpha} G}$ which is indeed an order zero map. 

Then for any $b\in A\rtimes_{\alpha}G$
\[ \begin {split}
(\widetilde{\psi} \circ \iota_{A\rtimes_{\alpha}G})(b)&= (\psi \circ \phi \circ \iota_{A\rtimes_{\alpha}G})(b)\\
&=(\psi \circ (1_{C(G)}\otimes \id_A)\rtimes G)(b)\\
&=bg.
\end{split} \]
Moreover, 
$1- \widetilde{\psi}(1)=1-\psi(1)=1-g \lesssim z$.  Therefore, we have shown that $\iota_{A\rtimes_{\alpha}G}:(A\rtimes_{\alpha}G, \widehat{\alpha}) \to ((A\rtimes_{\alpha}G)\rtimes_{\widehat{\alpha}}\widehat{G}, \Ad ({\lambda^{\widehat{\alpha}}})$ is $\widehat{G}$-tracially sequentially-split by order zero. Then by Corollary \ref {C:approxrepandmap} $\widehat{\alpha}$ has the weak tracial approximate representability. 

Conversely, let $B=A\rtimes_{\alpha}G$.  If $\widehat{\alpha}$ has the weka tracial  approximate representability, then $\iota_{B}:(B, \widehat{\alpha}) \to (B \rtimes_{\widehat{\alpha}}\widehat{G}, \Ad(\lambda^{\widehat{\alpha}}))$ is $\widehat{G}$-tracially sequentially-split by order zero. It means that for a nonzero positive element $z \in B_{\omega}$ we have  the following  diagram  with the tracial approximate inverse $\psi$
\begin{equation*}\label{D:diagram2}
\xymatrix{ (B, \widehat{\alpha}) \ar[rd]_{ \iota_B} \ar@{-->}[rr]^{ \iota} &&(B_{\omega}, \widehat{\alpha}_{\omega}))  &\\
                          &( B \rtimes_{\widehat{\alpha}} \widehat{G}, \Ad (\lambda^{\widehat{\alpha}}))  \ar[ur]_{\psi: \, \text{order zero}} }
\end{equation*}
Again combining this diagram with the diagram in Lemma \ref{L:BS1} we obtain the following diagram
 \begin{equation*}\label{D:diagram2}
\xymatrix{ (A\rtimes_{\alpha}G, \widehat{\alpha}) \ar[dd]_{(1_{C(G)}\otimes \,\id_A)\rtimes G \quad}\ar[rd]_{\iota_{A\rtimes_{\alpha}G}}\ar@{-->}[rr]^{ \iota} &&((A\rtimes_{\alpha}G)_{\omega}, (\widehat{\alpha})_{\omega})  &    \\
                       &((A\rtimes_{\alpha}G) \rtimes_{\widehat{\alpha}}\widehat{G},\Ad(\lambda^{\widehat{\alpha}}))  \ar[ur]_{\psi: \, \text{order zero}}  \\
((C(G)\otimes A)\rtimes_{\sigma\otimes \alpha}G, \widehat{\sigma \otimes \alpha})    \ar[ru]_{\phi^{-1}} }
\end{equation*} 
Then using the same argument as before we can show that $(1_{C(G)}\otimes \id_A)\rtimes G$ is $\widehat{G}$-tracially sequentially-split by order zero.  By Theorem \ref{T:dualityoftraciallysequentiallysplitmap}  the map $1_{C(G)}\otimes \id_A: (A, \alpha) \to (C(G)\otimes A, \sigma \otimes \alpha)$ is $G$-tracially sequentially-split by order zero. Thus $\alpha$ has the weak tracial Rokhlin property by Corollary \ref{C:thefirstfactorembeddingsplit}.\\

(2): Suppose that $\alpha$ has the weak approximate representability. Then $\iota_A: (A, \alpha) \to (A\rtimes_{\alpha}G, \Ad(\lambda^{\alpha}) )$ is $G$-tracially sequentially split by oder zero.  
By Theorem \ref{T:dualityoftraciallysequentiallysplitmap} $\iota_A \rtimes G: (A\rtimes_{\alpha} G, \widehat{\alpha}) \to ((A\rtimes_{\alpha}G)\rtimes_{\Ad(\lambda^{\alpha})}G, \widehat{\Ad \lambda^{\alpha}})$ is 
$\widehat{G}$-tracially sequentially-split by order zero. Thus for a nonzero positive element $z$ in $(A\rtimes_{\alpha} G)_{\omega}$ we have the following diagram with the tracial approximate left inverse $\psi$
 \begin{equation*}\label{D:diagram5}
\xymatrix{ (A\rtimes_{\alpha}G, \widehat{\alpha}) \ar[rd]_{\iota_A \rtimes G} \ar@{-->}[rr]^{ \iota} &&((A\rtimes_{\alpha}G)_{\omega}, (\widehat{\alpha})_{\omega})  &\\
                          &((A\rtimes_{\alpha}G)\rtimes_{\Ad(\lambda^{\alpha})}G, \widehat{\Ad \lambda^{\alpha}})  \ar[ur]_{\psi: \, \text{order zero}}}
\end{equation*}
Combining this diagram with the diagram in Lemma \ref{L:BS2} we also have the following diagram
 \begin{equation*}\label{D:diagram6}
\xymatrix{ (A\rtimes_{\alpha}G, \widehat{\alpha}) \ar[dd]_{1_{C(\widehat{G})}\otimes \id_{A\rtimes_{\alpha}G}}\ar[rd]_{\iota_A \rtimes G} \ar@{-->}[rr]^{ \iota} &&((A\rtimes_{\alpha}G)_{\omega}, (\widehat{\alpha})_{\omega})  &\\
                          &((A\rtimes_{\alpha}G)\rtimes_{\Ad(\lambda^{\alpha})}G, \widehat{\Ad \lambda^{\alpha}})  \ar[ur]_{\psi: \, \text{order zero}} \\
(C(\widehat{G})\otimes(A\rtimes_{\alpha}G), \sigma\otimes \widehat{\alpha})\ar[ru]_{\phi} }
\end{equation*} 
This means that the second factor embedding $1_{C(\widehat{G})}\otimes \id_{A\rtimes_{\alpha}G}:(A\rtimes_{\alpha}G, \widehat{\alpha}) \to (C(\widehat{G}) \otimes (A\rtimes_{\alpha}G), \sigma\otimes \widehat{\alpha})$ has the tracial approximate left inverse $\psi \circ \phi$ for $z$, thus $\widehat{G}$-tracially sequentially-split by order zero. Therefore,  $\widehat{\alpha}$ has the weak tracial Rokhlin property by Corollary \ref{C:thefirstfactorembeddingsplit}.\\

Conversely if $\widehat{\alpha}$ has the weak tracial Rokhlin property, the second factor embedding $1_{C(\widehat{G})}\otimes \id_{A\rtimes_{\alpha}G}:(A\rtimes_{\alpha}G, \widehat{\alpha}) \to (C(\widehat{G}) \otimes (A\rtimes_{\alpha}G), \sigma\otimes \widehat{\alpha})$ is $\widehat{G}$-tracially sequentially-split by order zero by Corollary \ref{C:thefirstfactorembeddingsplit}. Thus for a nonzero positive element $z$ in $(A\rtimes_{\alpha}G)_{\omega}$ we have the following diagram with the tracial approximate left inverse $\widetilde{\psi}$. 
 \begin{equation*}\label{D:diagram7}
\xymatrix{ (A\rtimes_{\alpha}G, \widehat{\alpha}) \ar[rd]_{1_{C(\widehat{G})}\otimes \id_{A\rtimes_{\alpha}G}} \ar@{-->}[rr]^{ \iota} &&((A\rtimes_{\alpha}G)_{\omega}, (\widehat{\alpha})_{\omega})  &\\
                          &(C(\widehat{G})\otimes(A\rtimes_{\alpha}G), \sigma\otimes \widehat{\alpha})  \ar[ur]_{\widetilde{\psi}: \, \text{order zero}}
                          }
\end{equation*} 
Then again combining  this diagram with the diagram in Lemma \ref{L:BS2} we have the following diagram
 \begin{equation*}\label{D:diagram8}
\xymatrix{ (A\rtimes_{\alpha}G, \widehat{\alpha}) \ar[dd]_{\iota_A\rtimes G} \ar[rd]_{1_{C(\widehat{G})}\otimes \id_{A\rtimes_{\alpha}G}} \ar@{-->}[rr]^{ \iota} &&((A\rtimes_{\alpha}G)_{\omega}, (\widehat{\alpha})_{\omega})  &\\
                          &(C(\widehat{G})\otimes(A\rtimes_{\alpha}G), \sigma\otimes \widehat{\alpha})  \ar[ur]_{\widetilde{\psi}:\, \text{order zero} } \\
((A\rtimes_{\alpha}G)\rtimes_{\Ad(\lambda^{\alpha})}G, \widehat{\Ad \lambda^{\alpha}}). \ar[ru]_{\psi^{-1}} }
\end{equation*}
This means that  the map $\iota_A \rtimes G: (A\rtimes_{\alpha}G, \widehat{\alpha}) \to ((A\rtimes_{\alpha}G)\rtimes_{\Ad(\lambda^{\alpha})}G, \widehat{\Ad \lambda^{\alpha}})$ has the tracial approximate left inverse $\widetilde{\psi}\circ \psi^{-1}$ for $z$, thus  $\widehat{G}$-tracially sequentially-split by order zero. Then Theorem \ref{T:dualityoftraciallysequentiallysplitmap} implies that $\iota_A: (A, \alpha) \to (A\rtimes_{\alpha}G, \Ad(\lambda^{\alpha}))$ is $G$-tracially sequentially-split by order zero. It follows from Corollary \ref{C:approxrepandmap} that $\alpha$ has the weak tracial approximate representability.
\end{proof}
\section{Acknowledgements}
This research is a by-product  of the author's collaboration with H. Osaka. He would like to thank him for allowing me to publish it independently.   

\end{document}